\documentclass[a4paper]{amsart} 
\usepackage{graphicx}
\usepackage{color}
\usepackage{hyperref}
\usepackage{tikz,amssymb,amsmath} 
\usepackage[arrow,curve,matrix,tips]{xy}
\usepackage{enumerate}
\usepackage{tkz-euclide}
\usepackage{esint}
\usepackage{setspace}
\usepackage{verbatim}
\makeatletter
\tikzset{
  treenode/.style = {align=center, inner sep=0pt, text centered,
    font=\sffamily},
  arn_n/.style = {treenode, circle, white, font=\sffamily\bfseries, draw=black,
    fill=black, text width=1.5em},
  arn_r/.style = {treenode, circle, red, draw=red, 
    text width=1.5em, very thick},
  arn_x/.style = {treenode, rectangle, draw=black,
    minimum width=0.5em, minimum height=0.5em}
}

\newtheorem{theorem}{Theorem}[section]
\newtheorem{lemma}[theorem]{Lemma}

\newtheorem{proposition}[theorem]{Proposition}
\newtheorem{corollary}[theorem]{Corollary}

\theoremstyle{definition} 
\newtheorem{definition}[theorem]{Definition}

\newtheorem{remark}[theorem]{Remark}

\numberwithin{equation}{section} 
\makeatletter
\def\moverlay{\mathpalette\mov@rlay}
\def\mov@rlay#1#2{\leavevmode\vtop{%
   \baselineskip\z@skip \lineskiplimit-\maxdimen
   \ialign{\hfil$\m@th#1##$\hfil\cr#2\crcr}}}
\newcommand{\charfusion}[3][\mathord]{
    #1{\ifx#1\mathop\vphantom{#2}\fi
        \mathpalette\mov@rlay{#2\cr#3}
      }
    \ifx#1\mathop\expandafter\displaylimits\fi}
\makeatother

\begin{document} 

\title{Quantifier elimination for quasi-real closed fields}

\author{Mickael Matusinski, Simon M\"uller}

\begin{abstract} 
We prove quantifier elimination for the theory of quasi-real closed fields with a compatible valuation. This unifies the same known results for algebraically closed valued fields and real closed valued fields.
\end{abstract}

\begin{center} {\small{\today}} \end{center}

\maketitle

\section{Introduction.}

\noindent Ordered fields and valued fields share several similar features, most notably via the notion of (order-)compatible valuations and the corresponding rank (see e.g. \cite{Zariski, Kuhlmatu, Kuhlmul} and below). In the present note we illustrate this by subsuming the following major model theoretic results: the (first order) theories of algebraically closed fields with a non-trivial valuation, respectively real closed fields with a \emph{compatible} non-trivial valuation, both admit quantifier elimination (cf. \cite{Robinson}, respectively \cite{Cherlin}). We prove that the same applies if we consider the theory of the union of these two classes in a common language, the one of quasi-ordered fields. Quasi-ordered fields were introduced in \cite{Fakh} and provide a uniform axiomatisation of ordered and valued fields. We derive a notion of quasi-real closed field, which subsumes the classes of algebraically closed fields and real closed fields. Note that real closed fields were developed by Artin and Schreier in \cite{ArtinSchreier} as maximal ordered fields, in analogy to the notion of algebraic closure of a field.

\section{Preliminaries on quasi-ordered fields.} \label{sect:qo-fields}

\noindent
Throughout this note, all orderings and quasi-orderings are always assumed to be total. An \emph{ordered field} is a field $K$ equipped with an ordering $\leq$ (i.e. a binary, reflexive, transitive  and anti-symmetric relation) such that the following axioms are satisfied for all $x,y,z \in K\colon$
\begin{itemize}
	\item[(O1)] $x \leq y\, \land\, 0 \leq z \ \Rightarrow\ xz \leq yz,$
	\item[(O2)] $x \leq y \ \Rightarrow\  x+z \leq y+z.$ 
\end{itemize} 

\noindent
Orderings are in one-to-one correspondence with so-called \emph{positive cones}, i.e. subsets $P \subseteq K$ such that $PP \subseteq P,$ $P+P \subseteq P,$ $P\cup(-P)=K$ and $-1\notin P$: $0 \leq x\Leftrightarrow x\in P$.
 
\noindent
A valuation $v$ on a field $K$ (in the sense of Krull \cite{Krull}) is known to endow $K$ with a \emph{quasi-ordering}, i.e.  a binary, reflexive, transitive relation $\preceq,$ as follows:
$$a \preceq b \Leftrightarrow v(b)\leq v(a).$$
If $\preceq$ is a quasi-ordering, we write $x \asymp y$ as shorthand for $x \preceq y \preceq x.$ Note that if $\preceq$ is an ordering, then $\asymp$ is just equality, whereas if $\preceq$ is induced by a valuation $v$ as above, then $x \asymp y$ if and only if $v(x) = v(y).$ Moreover, we write $x \prec y$ as shorthand for $x \preceq y$ and $x \not\asymp y.$

\enlargethispage{2mm}

\noindent
Pushing the analogy of ordered and valued fields further, Fakhruddin introduced the notion of \emph{quasi-ordered field}, that is a field $K$ equipped with a quasi-ordering $\preceq$ such that the following axioms are satisfied for all $x,y,z \in K\colon$
\begin{itemize}
	\item[(Q0)] $x \asymp 0 \ \Rightarrow \ x = 0,$
	\item[(Q1)] $x \preceq y \,\land\, 0 \preceq z \ \Rightarrow \ xz \preceq yz,$
	\item[(Q2)] $x \preceq y \,\land\, z \not\asymp y \ \Rightarrow \ x+z \preceq y+z.$
\end{itemize}

\begin{theorem}[Fakhruddin's dichotomy]  \label{dichotomy} \cite[Theorem 2.1]{Fakh} Let $K$ be a field and $\preceq$ a binary relation on $K$. Then $(K,\preceq)$ is a quasi-ordered field if and only if it is either an ordered field or there is a valuation $v$ on $K$ such that $x \preceq y \Leftrightarrow v(y) \leq v(x)$ for all $x,y \in K.$
\end{theorem}

\noindent 
Note that $\preceq$ is induced by a valuation $v$ on $K$ if and only if $0 \prec -1.$ In that case $v$ is unique up to equivalence of valuations, and we denote the quasi-ordering $\preceq$ also by $\preceq_v.$

\begin{remark}  In \cite{Efrat}, a \emph{locality} on a field is, by definition, either an ordering or a valuation. Hence, by Fakhruddin's dichotomy, localities and quasi-orderings are in fact equivalent objects.
\end{remark}

\noindent
Baer and Krull exhibited another strong relation between orderings and valuations on fields \cite{Baer,Krull}: any ordering $\leq$ on $K$ induces a valuation on $K,$ the so-called \emph{natural valuation} of $(K,\leq).$ It is the valuation on $K$ whose valuation ring is the convex hull of $\mathbb{Z}$ in $K,$ i.e. the smallest convex subring of $K$ w.r.t. $\leq.$ An ordering $\leq$ and a valuation $v$ on $K$ are called \emph{compatible}, if $K_v$ is convex. We also say that $v,$ respectively $K_v,$ is convex. In \cite{Kuhlmatu}, the authors generalised this notion to quasi-orderings: given a quasi-ordered field $(K,\preceq)$, a valuation $v$ on $K$ -- or equivalently a quasi-ordering $\preceq_v$ -- is said to be \emph{compatible with $\preceq$}, if the valuation ring $K_v$ is $\preceq$-convex, or equivalently, if $0 \preceq x\preceq y\, \Rightarrow\, x\preceq_v y$ for all $x,y \in K.$ Note that if $\preceq$ is induced by a valuation $w,$ then $\preceq_v$ is compatible with $\preceq_w$ if and only if $v$ is a coarsening of $w.$ For a further discussion on this subject we refer the interested reader to \cite{Kuhlmatu} and \cite{Kuhlmul}. \vspace{1mm}

\noindent 
We conclude this section by recalling real closed fields. An ordered field $(K,\leq)$ is said to be \emph{real closed} if one of the following two equivalent conditions is satisfied:

\begin{itemize}
	\item[(RC1)] $\leq$ does not extend to any proper algebraic extension of $K.$
	\item[(RC2)] $(K,\leq)$ satisfies the following conditions:
	\begin{itemize}
		\item[(i)] any $x \in K$ with $0 \preceq x$ has a square root in $K,$
		\item[(ii)] any polynomial $f \in K[X]$ of odd degree has a root in $K.$
	\end{itemize}
\end{itemize}

\noindent
Any ordered field $(K,\leq)$ admits a \emph{real closure}, i.e. an algebraic field extension which is real closed. It is unique up to isomorphism of ordered fields (\cite[Satz 8]{ArtinSchreier}).

\section{Quasi-real closed fields}

\noindent
In the present section we introduce quasi-real closed fields, that way unifying algebraically closed fields and real closed fields. To this end, we impose (RC1) on quasi-ordered fields.


\begin{definition}\label{qrcfield1} A quasi-ordered field $(K,\preceq)$ is called \emph{quasi-real closed}, if $\preceq$ does not extend to any proper algebraic extension of $K.$
\end{definition}

\begin{lemma} \label{acfield} Let $K$ be a field. The following are equivalent:
	\begin{enumerate}
		\item $K$ is algebraically closed,
		\item $(K,\preceq_v)$ is quasi-real closed for some valuation $v$ on $K,$ 
		\item $(K,\preceq_v)$ is quasi-real closed for any valuation $v$ on $K.$
	\end{enumerate}
\end{lemma}
\begin{proof} If $K$ is algebraically closed, then $K$ admits no proper algebraic extension, whence $(K,v)$ is quasi-real closed for any valuation $v$ on $K.$ Consequently, (1) implies (3). Moreover, since any field admits a valuation, clearly (3) implies (2). Finally, suppose that (2) holds. Note that $v$ extends to any field extension $L$ of $K$ by an immediate consequence of Chevalley's Extension Theorem \cite[Theorems 3.1.1 and 3.1.2]{Prestel}. Since $(K,v)$ is quasi-real closed, this means that $K$ does not admit any proper algebraic extension, i.e. that $K$ is algebraically closed. Thus, (2) implies (1).
\end{proof}

\begin{theorem} \label{dichoclosed} A quasi-ordered field $(K,\preceq)$ is quasi-real closed if and only if one of the following two equivalent conditions holds:
	\begin{description}
		\item[(QRC1)]  either $\preceq$ is an ordering and $(K,\preceq)$ is a real closed field, or $\preceq$ is induced by a valuation and $K$ is an algebraically closed field;
		\item[(QRC2)] $(K,\preceq)$ satisfies the following conditions:
		\begin{enumerate}
			\item[(i)] any $x \in K$ with $0 \preceq x$ has a square root in $K.$
						\item[(ii)] any polynomial $f \in K[X]$ of odd degree has a root in $K.$
		\end{enumerate}
	\end{description} 
\end{theorem}
\begin{proof} (QRC1) is an immediate consequence of Fakhruddin's dichotomy (Theorem \ref{dichotomy}) and (RC1) (if $\preceq$ is an ordering), respectively Lemma \ref{acfield} (if $\preceq$ is induced by a valuation).

\noindent
Next, we consider (QRC2). If $\preceq$ is an ordering, this is precisely (RC2). So let $\preceq$ be induced by some valuation. If $(K,\preceq)$ is quasi-real closed, then $K$ is algebraically closed according to Lemma \ref{acfield}, whence (i) and (ii) are both fulfilled. Conversely, suppose that (i) and (ii) hold. Since $K=\{0 \preceq x\}$, we obtain by (i) that all elements in $K$ have a square root in $K.$ 
	Along with condition (ii), 
	 this implies that $K$ is algebraically closed according to \cite[Theorem 2]{Shipman}. Hence, $(K,\preceq)$ is quasi-real closed again by Lemma \ref{acfield}.
\end{proof}

\begin{definition} Let $(K,\preceq)$ be a quasi-ordered field. A quasi-ordered field $(L,\preceq')$ is called a \emph{quasi-real closure} of $(K,\preceq),$ if
	
	\begin{enumerate}
		\item $(L,\preceq')$ is a quasi-real closed field,
		
		\item $L|K$ is an algebraic field extension,
		
		\item $\preceq'$ is an extension of $\preceq,$ i.e. $K \, \cap \preceq' \, = \, \preceq.$
	\end{enumerate}
\end{definition}



\begin{proposition} \label{qrclosure} 
	Any quasi-ordered field $(K,\preceq)$ admits a quasi-real closure. It is unique up to isomorphism of quasi-ordered fields.
\end{proposition}
	\begin{proof} Let us consider an algebraic closure $\tilde{K}$ of $K$ and  the set 
		$$\left\{ (L,\preceq')\colon (L,\preceq') \textrm{ is a quasi-ordered field},\ K\subseteq L\subseteq \tilde{K},\ K\cap \preceq' \, = \, \preceq \right\},$$
		partially ordered by
		$$(L_1,\preceq_1') \leq (L_2,\preceq_2') :\Leftrightarrow L_1 \subseteq L_2 \textrm{ and } L_1 \, \cap \preceq_2' \; = \; \preceq_1'.$$
	By Zorn's lemma, there is a  maximal quasi-ordered algebraic field extension  $(L,\preceq')$ of $(K,\preceq)$. Hence, $(L,\preceq')$ is quasi-real closed and a quasi-real closure of $(K,\preceq).$ 

\noindent
The proof of the uniqueness relies on Fakhruddin's dichotomy and (QRC1). So let $(K_1,\preceq_1)$ and $(K_2,\preceq_2)$ be quasi-real closures of $(K,\preceq)$. If $\preceq$ is an ordering, then $(K_1,\preceq_1)$ and $(K_2,\preceq_2)$ are real closures of $(K,\preceq),$ whence they are order-isomorphic according to \cite[Satz 8]{ArtinSchreier}. 

\noindent
Likewise, if $\preceq $ is induced by some valuation on $K,$ then $K_1$ and $K_2$ are algebraic closures of $K.$ Hence, they are isomorphic as fields. Let $\preceq_1 = \preceq_{v_1}$ and $\preceq_2 = \preceq_{v_2}.$ Then the corresponding valuation rings $K_{v_1}$ and $K_{v_2}$ are conjugated by an element of $\mathrm{Gal}(\tilde{K}|K)$ (\cite[Conjugation Theorem 3.2.15]{Prestel}). By composition, we obtain an isomorphism between $(K_1,\preceq_1)$ and $(K_2,\preceq_2)$.
\end{proof}

\section{The theory of quasi-real closed fields and quantifier elimination}

\noindent
We conclude this note by proving that the theory of quasi-real closed fields with a non-trivial compatible valuation admits quantifier elimination. Moreover, we deduce that this theory is the model companion of the theory of quasi-ordered fields with a non-trivial compatible valuation. To this end, we exploit Fakhruddin's dichotomy and the fact that the theories of real closed valued fields and algebraically closed valued fields both admit quantifier elimination. 

\begin{theorem} \label{qeorval} $\mathrm{(cf.}$ \cite[Theorems 4.4.2 and 4.5.1]{Prestel3}$\mathrm{)}$
	\begin{enumerate}
		
		\item The theory of algebraically closed fields with a non-trivial valuation $v$ admits quantifier elimination in the language of fields adjoined with $\preceq_v.$
		
		\item The theory of real closed fields with a non-trivial compatible valuation $v$ admits quantifier elimination in the language of ordered fields adjoined with $\preceq_v.$
		
	\end{enumerate}
\end{theorem}

\noindent
The characterization (QRC2) from Theorem \ref{dichoclosed} allows us to formalise quasi-real closed fields in a first order language. As language of quasi-ordered fields with a compatible valuation $v$, we fix $\mathcal{L} = \{+,\cdot,-,\preceq,\preceq_v,0,1\}.$ Moreover, we denote by $\Sigma_{\rm QRC}$ the following set of $\mathcal{L}$-sentences:

\begin{enumerate}
	\item the axioms that $(K,+,\cdot,-)$ is a field.
	
	\item the axioms that $(K,\preceq)$ is a quasi-real closed field.
	
	\item the axioms that $(K,\preceq_v)$ is a quasi-ordered field with $v$ non trivial.
	
	\item the following $\mathcal{L}$-sentences that determine the relationship of $\preceq$ and $\preceq_v:$
	
	\begin{itemize}
		\item[(i)] $\phi_1 \equiv 0 \prec -1 \rightarrow (\forall x,y \colon 0 \preceq x \preceq y \leftrightarrow x \preceq_v y)$ \vspace{1mm}
		
		\item[(ii)] $\phi_2 \equiv -1 \prec 0 \rightarrow (\forall x,y \colon 0 \preceq x \preceq y \rightarrow x \preceq_v y)$ \vspace{1mm}
	\end{itemize}
\end{enumerate}

\noindent
The $\mathcal{L}$-sentences $\phi_1$ and $\phi_2$ state that $\preceq \, = \, \preceq_v$ if $\preceq$ is induced by a valuation, and that $v$ is $\preceq$-compatible if $\preceq$ is an ordering. That way, we have unified the languages of real closed valued fields and algebraically closed valued fields that are used in Theorem \ref{qeorval}. We refer to the theory of $ \Sigma_{\textrm{QRC}}$ as the \textit{theory of quasi-real closed fields with a non-trivial compatible valuation}. We also consider the \textit{theory of quasi-ordered fields with a non-trivial compatible valuation}, denoted by $\Sigma_{\rm QO},$ which consists of (1), (3), (4), and the axioms (Q0), (Q1) and (Q2) (see Section \ref{sect:qo-fields}).

\begin{theorem} \label{thmqe} The theory of \, $\Sigma_{\textrm{QRC}}$ admits quantifier elimination.
\end{theorem}
\begin{proof} Any model of $\Sigma_{\textrm{QRC}}$ is either an algebraically closed field with a non-trivial valuation or a real closed field with a non-trivial compatible valuation (Corollary \ref{dichoclosed}). Therefore, the result follows from the fact that each of these two classes admits quantifier elimination (Theorem \ref{qeorval}).
\end{proof}

\begin{definition} \label{companion} Let $\Sigma \subseteq \textrm{Sent}(\mathcal{L})$ for a given language $\mathcal{L}.$ A set $\Sigma^* \subseteq \textrm{Sent}(\mathcal{L})$ is called a \textit{model companion} of $\Sigma,$ if
	\begin{enumerate}
		\item every model of $\Sigma^*$ is a model of $\Sigma,$
		
		\item every model of $\Sigma$ can be extended to a model of $\Sigma^*,$
		
		\item $\Sigma^*$ is model complete.
	\end{enumerate}
\end{definition}

\begin{corollary} $\mathrm{(cf.}$ \cite[Corollaries 4.4.3 and 4.5.4]{Prestel3}$\mathrm{)}$ \label{corqe} \\ The theory of $\, \Sigma_{\textrm{QRC}}$ is model complete. It is the model companion of the theory of $\, \Sigma_{\textrm{QO}}$.
\end{corollary}
\begin{proof} Property (1) is obviously satisfied. The model completeness follows immediately from the quantifier elimination that we obtained in Theorem \ref{thmqe}. It remains to show that condition (2) of Definition \ref{companion} holds in our setting.
	
	\noindent
	So let $(K,\preceq,\preceq_v)$ be some quasi-ordered field with a non-trivial compatible valuation, and let $(L,\preceq')$ be a quasi-real closure of $(K,\preceq)$ (Proposition \ref{qrclosure}). Then $(L,\preceq')$ is a quasi-real closed field. Moreover, $L|K$ is algebraic, whence \cite[Theorem 20.1.1(a)]{Efrat} yields a unique extension $\preceq'_{\nu}$ of $\preceq_v$ from $K$ to $L$ such that $\nu$ is compatible with $\preceq'.$ If $\preceq$ is a valuation (i.e. $\preceq_v \, = \, \preceq$), then the uniqueness tells us that also $\preceq'_{\nu} \, = \, \preceq'.$ Furthermore, $\nu$ is non-trivial since $v$ is non-trivial. Hence, $(L,\preceq',\preceq'_{\nu})$ is an extension of $(K,\preceq,\preceq_v)$ and a model of the theory of quasi-real closed fields with a non-trivial compatible valuation.
\end{proof}

\begin{remark}
The theory of $\, \Sigma_{\textrm{QRC}}$ is not complete. For example the $\mathcal{L}$-sentence $\exists x \colon x^2 +1 = 0$ is false for any real closed field, but true for any algebraically closed field.
\end{remark} 

\noindent
\textbf{Open Question:} Can we replace the $\mathcal{L}$-sentences $\phi_1$ and $\phi_2$ in $\Sigma_{\textrm{QRC}}$ with the single $\mathcal{L}$-sentence $\phi \equiv \forall x,y \colon x \preceq y \rightarrow x \preceq_v y?$ This is equivalent to the question, whether algebraically closed fields admit quantifier elimination in the language of fields adjoined by $\preceq_v$ and $\preceq_w,$ where $v$ and $w$ are valuations on $K$ such that $v$ is coarser than $w.$

\section*{Acknowledgment} 
\noindent
The authors wish to thank Salma Kuhlmann for initiating and supporting their collaboration, and also for helpful discussion on the subject. Without her, this paper would not have come to pass.

 \vspace{4mm}

\small{

\noindent
\textsc{IMB, Universit\'{e}	 Bordeaux 1, 33405 Talence, France}, \\ 
\textit{E-mail address:} mickael.matusinski@math.u-bordeaux.fr \vspace{2mm}

\noindent
\textsc{Uni Konstanz, FB Mathematik und Statistik, 78457 Konstanz, Germany}, \\ 
\textit{E-mail address:} simon.2.mueller@uni-konstanz.de}

\end{document}